\documentclass{elsarticle}
\usepackage[english]{babel}
\usepackage[utf8x]{inputenc}
\usepackage[T1]{fontenc}
\usepackage{graphicx}
\usepackage[colorinlistoftodos]{todonotes}
\usepackage[colorlinks=true, allcolors=tudelftblue]{hyperref} 
\usepackage{amssymb}
\usepackage{caption}
\usepackage{subcaption}
\usepackage{xcolor}
\usepackage{float}
\usepackage{algorithm}
\usepackage[noend]{algpseudocode}
\usepackage{titling} 
\usepackage{blindtext}\usepackage{titlesec}
\usepackage[colorinlistoftodos]{todonotes}
\usepackage{tikz}
\usepackage{array}
\usepackage{geometry}
\usepackage{mathtools}
\usepackage{hyperref} 
\usepackage{sectsty}
\usepackage{amsmath}
\usepackage{listings}  
\usepackage{color}     
\usepackage{pgfplots}
\pgfplotsset{compat=newest}
\usepackage{amsthm}
\newtheorem{definition}{Definition}
\newtheorem{theorem}{Theorem}
\newtheorem{lemma}{Lemma}
\newtheorem{proposition}{proposition}
\newtheorem{assumption}{Assumption}
\newtheorem{remark}{Remark}
\theoremstyle{remark}
\newtheorem{example}{Example}
\usepackage{enumitem}
\allowdisplaybreaks
\usepackage{tikzpagenodes}
\usepackage[justification=centering]{caption}
\usepackage{booktabs}
\usepackage{listings}
\numberwithin{equation}{section}
\definecolor{tudelftdarkblue}{RGB}{12,35,64}
\definecolor{tudelftcyan}{RGB}{0,166,214}
\definecolor{tudelftblue}{RGB}{0, 118, 194}
\allsectionsfont{\color{black}} 
\usepackage{accents}

\newcommand{\customfootnotetext}[2]{{
  \renewcommand{\thefootnote}{#1}
  \footnotetext[0]{#2}}}
\titleformat{\paragraph}
{\normalfont\normalsize\bfseries}{\theparagraph}{1em}{}
\titlespacing*{\paragraph}
{0pt}{3.25ex plus 1ex minus .2ex}{1.5ex plus .2ex}
\usepackage{setspace}
\begin{document}

\newpage
\begin{center}\customfootnotetext{$\dagger$}{Electronic Supplementary Information (ESI) available}
{\LARGE A Cubic Underestimator for Deterministic Global Optimization\textsuperscript{$\dagger$}}
\end{center}

\begin{center}
{\textit{Krishan Kumar and Ishan Bajaj\textsuperscript{*}\customfootnotetext{*}{Corresponding author: email: \href{mailto:ibajaj@iitk.ac.in}{ibajaj@iitk.ac.in}}}}
\\
{\small \textit{Department of Chemical Engineering, Indian Institute of Technology Kanpur, Uttar Pradesh, India, 208016}}
\addcontentsline{toc}{section}{Abstract}
\section*{Abstract}
\begin{quote}
We propose a new underestimator with pointwise and Hausdorff convergence of order three for deterministic global optimization of twice continuously differentiable functions with Lipschitz continuous Hessians. The proposed underestimator is obtained by augmenting the second-order Taylor expansion with a cubic remainder term of the objective function yielding a valid global underestimator that can be used to obtain lower bounds in a branch-and-bound method to solve nonconvex global optimization problems. The underestimator is nonconvex, and finding its global minimum could be computationally expensive. Accordingly, we derive a sufficient condition under which the proposed cubic underestimator is edge-concave, so that its minimum over a box can be found by vertex evaluation, while retaining third-order convergence. We show that the maximum separation distance between the objective function and the underestimator decreases cubically with the width of the box. This higher order convergence enables the lower bounds to converge more rapidly to the optimal value as the variable bounds are progressively tightened. The proposed underestimator is incorporated into a deterministic branch-and-bound algorithm and compared with the $\alpha$BB and second-order edge-concave underestimators on a collection of 130 benchmark problems. The numerical results show that the proposed underestimator outperforms the compared underestimators in terms of both the number of explored branch-and-bound nodes and total CPU time, while maintaining a competitive computational cost per node. These results demonstrate that the proposed cubic edge-concave underestimator can provide more effective lower bounds and better computational performance in deterministic global optimization.
\end{quote}
\end{center}
\textbf{Keywords:} Cubic underestimator, Edge-concave functions, Deterministic global optimization, Branch-and-bound
\section{Introduction}\label{sec1}

We consider the following box-constrained global optimization problem
\[
\begin{aligned}
\min_{x\in P} &~ f(x),
\end{aligned}
\]
where $P=[x^L,x^U]\subset\mathbb{R}^n$ and $f:P\rightarrow\mathbb{R}$ is a twice continuously differentiable nonconvex function whose Hessian is Lipschitz continuous. Such problems arise in many areas of science and engineering, including machine learning, process systems engineering, economics, and data science. Since a nonconvex function may have many local minima, finding the global minimum is often computationally challenging (it is, in fact, NP hard \cite{vavasis1991nonlinear}). Deterministic global optimization methods enable finding the global minimum while providing rigorous guarantees of global optimality \cite{al1990jointly, horst1987convergence, hansen1979global, horst1992new, tuy1985conical,hansen1992global,benson1982convergence,ratz1995selection}. Branch-and-bound methods are one of the most widely used frameworks for global optimization. These methods recursively partition the feasible region and associate with each subregion a lower bound on the objective function and an upper bound obtained from a feasible solution. A subregion can then be discarded whenever its lower bound is larger than the best available upper bound. Thus, the efficiency of a branch-and-bound algorithm largely depends on the quality of the lower bounds computed at each node. Therefore, constructing tight and computationally efficient underestimators is a fundamental step in the development of effective deterministic global optimization algorithms \cite{floudas2013deterministic}.

Convex relaxations have played a fundamental role in deterministic global optimization. In particular, the convex envelope is the tightest convex underestimator of a function over a given domain, but explicit convex envelope representations are available only for certain classes of functions such as bilinear \cite{al1983jointly}, trilinear \cite{meyer2004trilinear,meyer2004trilinear1}, multilinear \cite{ryoo2001analysis}, univariate monomials of odd degree \cite{liberti2003convex}, and low-dimensional edge-concave functions \cite{tardella1988class,meyer2005convex,tardella2008existence}. 

McCormick relaxations are an important class of relaxations for deterministic global optimization. For factorable functions, McCormick relaxations provide a systematic way to construct convex underestimators and concave overestimators by applying relaxation rules to basic operations such as addition, multiplication, and composition \cite{mccormick1976computability,bompadre2012convergence}. Tsoukalas and Mitsos \cite{tsoukalas2014multivariate} further generalized the framework to multivariate outer functions and provided relaxation rules for operations such as multiplication, minimum, maximum, and fractional functions. McCormick relaxations have also been used in global optimization problems with embedded differential equations \cite{mitsos2009mccormick}. For further developments based on McCormick relaxations, readers may refer to \cite{scott2011generalized,sahlodin2011convex,scott2013improved,wechsung2015reverse,stuber2015convex,khan2017differentiable,ye2023extended} and the references therein.

When the objective function is nonconvex and does not possess an exploitable mathematical form, general underestimators, such as $\alpha$BB and edge-concave underestimators, have been proposed  \cite{maranas1994global,androulakis1995alphabb,adjiman1996rigorous,adjiman1998global,hasan2018edge,nath2025separable}. The $\alpha$BB method constructs a convex underestimator by adding a quadratic correction term to the original nonconvex function. The parameter $\alpha$ is chosen to ensure convexity of the resulting relaxation over the prescribed domain. Similar to $\alpha$BB, other convex underestimators were also developed, where an exponential form is used instead of a quadratic expression, see \cite{akrotirianakis2004new,akrotirianakis2004computational}. Several methods have been proposed to improve the performance of $\alpha$BB. We refer the reader to \cite{skjal2012generalization, hladik2015efficient,guzman2016performance,nerantzis2019tighter,zerrouki2024tighter} and references therein.

Although convex underestimators provide an attractive framework for obtaining valid lower bounds, convexity is not the only structural property that can be exploited in global optimization. Depending on the domain of the problem under consideration, one underestimator may be tighter than another \cite{papamichail2002rigorous}.  An alternative approach is based on {edge-concavity}. An important consequence of edge-concavity is that an edge-concave function admits a vertex-polyhedral convex envelope \cite{tardella2003polyhedral}.

Building on this idea, Hasan \cite{hasan2018edge} proposed an edge-concave underestimator ffor general twice continuously differentiable nonconvex functions, constructed by subtracting a positive quadratic term from the objective. The proposed edge-concave underestimator provides a tight relaxation for many nonconvex problems and has been successfully applied to global dynamic optimization \cite{bajaj2020global}, black-box optimization \cite{bajaj2020deterministic}, and the approximation of phase equilibria \cite{iftakher2022data}. More recently, a separable edge-concave underestimator was proposed for signomial functions, whose convex envelope can be represented by a single linear hyperplane using only $n+1$ vertices \cite{nath2025separable}.

Besides the validity and tightness of an underestimator, another fundamental property of the associated estimator scheme is its convergence order. The convergence order determines how rapidly the relaxation error, defined as the maximum difference between the objective function and its underestimator over a search region, decreases as the search region becomes smaller and can therefore influence the computational effort required by the global optimization algorithm \cite{schobel2010theoretical}. Convergence order of an estimator scheme has been studied in the analysis of interval extensions \cite{moore1979methods,alefeld2000interval} and has subsequently been extended to various relaxation schemes \cite{bompadre2012convergence,najman2016convergence,schaber2019convergence}. In particular, the convergence order of the estimator schemes associated with McCormick relaxations was analysed by Bompadre and Mitsos \cite{bompadre2012convergence}. They studied how the convergence order is affected by basic operations, including addition, multiplication, and univariate composition in terms of both pointwise sense and in Hausdorff metric. Further, pointwise and Hausdorff convergence were investigated by Najman and Mitsos \cite{najman2016convergence} for an estimator scheme associated with multivariate McCormick relaxations. It was also shown that for general smooth nonlinear functions, the maximal order of pointwise convergence of the estimator is two, and it is achieved, for example, by $\alpha$BB relaxation and convex envelopes \cite{bompadre2012convergence}.

Taylor models are another important class of higher order estimators. A $q$th-order Taylor model combines a Taylor polynomial of degree $q$ with a rigorous bound on the remainder term. Under suitable assumptions, the resulting estimator schemes can exhibit convergence orders related to the order of the Taylor expansion. In particular, estimators of $q$th-order Taylor and McCormick-Taylor models have been shown to possess convergence order of at least $q+1$ \cite{bompadre2013convergence}.
These results suggest that higher order approximation models can provide more rapidly decreasing relaxation errors as the variable bounds becomes smaller. Despite their favorable convergence properties, Taylor models are not yet used in global optimization solvers because the polynomial part of a Taylor model is generally nonconvex, and computing its global minimum over a box is NP-hard in general \cite{bompadre2013convergence}. Bounding the polynomial part using convex relaxations or natural interval extensions can reduce the effective convergence order, potentially negating the higher-order advantage of the Taylor expansion. These computational challenges motivate the development of a novel underestimator that achieves higher-order convergence while allowing a tractable lower bounding computation within a branch-and-bound framework. 

In this work, we construct a higher-order underestimator based on the cubic remainder bound associated with a second-order Taylor expansion. The use of cubic term in second-order Taylor expansion has been studied by Nesterov and Polyak \cite{nesterov2006cubic} in the context of cubic regularization of Newton's method. However, our objective is different. Rather than using a cubic term to construct a model for determining the next iterate, we use the cubic remainder bound to construct a valid underestimator for deterministic global optimization. In particular, unlike the cubic regularization model in \cite{nesterov2006cubic}, the proposed underestimator incorporates a negative cubic term to provide a valid lower bound over a given box. This cubic term arises naturally from the Lipschitz continuity of the Hessian and provides a bound on the remainder of the second-order Taylor expansion. The proposed cubic underestimator possesses two features that are particularly relevant to deterministic global optimization. First, the maximum estimation error between the objective function and the underestimator is bounded by a cubic function of the distance from the expansion point. Second, by introducing a parameter $\Theta$ (analogous to $\alpha$ in $\alpha$BB \cite{androulakis1995alphabb} and $\theta$ in edge-concave underestimators \cite{hasan2018edge}) and selecting it appropriately, the proposed cubic underestimator can be made edge-concave over the prescribed box. This property enables the lower bound to be computed by evaluating the underestimator at the vertices of the box.

The main contributions of this work are summarized as follows.
\begin{itemize}[label={\scriptsize\textbullet}]
    \item We develop a cubic underestimator for twice continuously differentiable functions with Lipschitz continuous Hessians. The underestimator is obtained by using a cubic bound on the remainder of the second-order Taylor expansion.

    \item We derive a sufficient condition under which the proposed cubic underestimator is edge-concave on a box. Consequently, its minimum over the box can be obtained by evaluating the underestimator at the vertices, which provides a convenient lower bound for the branch-and-bound procedure.

    \item We study the convergence properties of the proposed cubic estimator scheme. In particular, we establish that the scheme has pointwise convergence of order three and that the corresponding interval enclosures have Hausdorff convergence of order three. We further derive a bound on the maximum separation between the function and the underestimator.

    \item We implement the proposed cubic edge-concave underestimator (CECU) within a deterministic branch-and-bound framework to find the global minimum of a nonconvex function subject to box constraints. The computational performance of the proposed underestimator is tested on a collection of 130 nonconvex test problems and compared with an $\alpha$BB method \cite{adjiman1998global} and the existing second-order edge-concave underestimator (SOEC) \cite{hasan2018edge}.
\end{itemize}

The remainder of this paper is organized as follows. Section \ref{sec:preliminaries} introduces the required notations and preliminaries. Section \ref{sec:cubic_underestimator} presents the construction and validity of the proposed cubic underestimator along with a sufficient condition for edge-concavity of the proposed underestimator. Section \ref{sec:convergence} develops the pointwise and Hausdorff convergence orders of the cubic estimator scheme. Section \ref{sec:bnb} describes the branch-and-bound algorithms based on the proposed underestimator and discusses how the difference between the value of the objective function and the underestimator decreases as the branch-and-bound subboxes are successively refined. Numerical results and comparisons with existing $\alpha$BB and SOEC are presented in Section \ref{sec:numerical}, followed by concluding remarks in Section \ref{sec:conclusion}.

\section{Preliminaries and notations}\label{sec:preliminaries}
In this section, we introduce the notation and basic definitions used throughout the paper. We introduce notation for the Euclidean space, inner product, and norm, followed by the gradient and Hessian. We then define boxes and their associated width, midpoint, and vertex set. 

We denote the set of real numbers by $\mathbb{R}$ and the $n$-dimensional Euclidean space by $\mathbb{R}^{n}$, that is,
\[
\mathbb{R}^{n} = \underbrace{\mathbb{R} \times\mathbb{R} \times \cdots \times\mathbb{R}}_{n \text{ times}}.
\]
For $x\in\mathbb{R}^n$, $x_i$ denotes its $i$th component.  The inner product and  Euclidean norm in $\mathbb{R}^{n}$ for any two vectors $x,y\in\mathbb{R}^n$ are taken as
\[
\langle x,y\rangle = x^{\top}y, \text{ and } \|x\|=\sqrt{\langle x,x\rangle}, \text{ respectively}.
\]
For a matrix $A\in\mathbb{R}^{n\times n}$, we denote its spectral norm by
\[
\Vert A\rVert_2 = \max_{\|x\|=1}\|Ax\|.
\]
For a twice continuously differentiable function $f:\mathbb{R}^{n}\rightarrow\mathbb{R}$, we denote its gradient and Hessian at $x\in\mathbb{R}^{n}$ by $\nabla f(x)$ and $\nabla^2f(x)$, respectively, where
\begin{equation}\label{eq:Grad}
\nabla f(x)=\left[\tfrac{\partial f}{\partial x_1}(x),\,
\tfrac{\partial f}{\partial x_2}(x),\,\ldots,\,\tfrac{\partial f}{\partial x_n}(x) \right]^{\top},
\end{equation}
\text{ and }
\begin{equation}\label{eq:Hessian}
\nabla^2 f(x) =[a_{lj}]_{n\times n} 
\text{ with } 
     a_{lj} = 
    \begin{cases}
    \dfrac{\partial^2 f}{\partial x_l\partial x_j}(x) & \text{ if } l \neq j, \\
    \dfrac{\partial^2 f}{\partial x_l^2}(x) & \text{ if } l=j.
    \end{cases}
\end{equation}
 For two vectors $x,y\in\mathbb{R}^n$, the relation $x\leq y$ is considered componentwise, i.e., $x\leq y \iff x_{i}\leq y_{i}$ for all $i=1,2,\ldots,n$. Let $ P=[x^L,x^U]= \{x\in\mathbb{R}^n:x^L\leq x\leq x^U\}$ be a box, where $x^L,x^U\in\mathbb{R}^n$ and $x_i^L\leq x_i^U$ for $i=1,2,\ldots,n$. For a box $P$, the width or diameter of $P$ is taken as
\[
\omega(P)=\operatorname{diam}(P)  =\sup_{x,y\in P}\|x-y\|,
\]
and the midpoint of $P$ is denoted by $x^{M}$ such that $x_{i}^M=\frac{x_{i}^L+x_{i}^U}{2} \text{ for each } i=1,2,\ldots, n$.
We denote by $V(P)$ the set of vertices of $P$, given by
\begin{equation*}
    V(P)=\left\{v\in\mathbb{R}^n: v_i\in\{x_i^L,x_i^U\}, \text{ for all } i=1,2,\ldots,n \right\}.
\end{equation*}
Thus, when $x_i^L<x_i^U$ (the case we consider in this paper) for all $i=1,2,\ldots,n$, the set $V(P)$ contains exactly $2^n$ elements, corresponding to all possible combinations of its lower and upper bounds.
\section{Construction of cubic underestimator}\label{sec:cubic_underestimator}
In this section, we construct the proposed cubic underestimator based on the second-order Taylor expansion of the objective function. Under Assumption \ref{ass:lipschitz}, we first establish a bound on the Taylor remainder. We then use this bound to define the cubic underestimator and its associated overestimator, and establish their validity over the box $P$.


\begin{assumption}\label{ass:lipschitz}
The function $f:P\rightarrow\mathbb{R}$
is twice continuously differentiable, and its Hessian matrix satisfies
\begin{equation*}\label{eq:HessianLip}
\left \lVert\nabla^2f(x)-\nabla^2f(y)\right\lVert_2 \le L\|x-y\| \text{ for all } x,y\in P,
\end{equation*}
where $L>0$ is the Lipschitz constant of the Hessian of $f$ on $P$.
\end{assumption}

Assumption \ref{ass:lipschitz} implies that the third-order Taylor remainder of the objective function admits a cubic bound in terms of the Lipschitz constant on the box set. The following lemma summarizes the properties that will be required in the subsequent analysis.

\begin{lemma}\label{lem:taylor}
Suppose Assumption \ref{ass:lipschitz} holds.
Then, for every $x,y\in P$,
\begin{equation}
\left\| \nabla f(y)-\nabla f(x) -\nabla^2f(x)(y-x) \right\| \le \frac{L}{2}\|y-x\|^2,
\label{eq:gradbound}
\end{equation}
and
\begin{equation}
\left\lvert f(y)-f(x) -\langle\nabla f(x),y-x\rangle -\frac12 \langle \nabla^2f(x)(y-x), y-x \rangle \right \rvert \le \frac{L}{6}\|y-x\|^3.
\label{eq:taylorbound}
\end{equation}
\end{lemma}
\begin{proof}
Consider a function $z: [0,1]\rightarrow \mathbb{R}^{n}$ such that
\[
z(t)=x+t(y-x) \text{ for all } t\in[0,1].
\]
Since $z(0)=x$ and $z(1)=y$, therefore the fundamental theorem of calculus gives
\begin{equation}\label{Lm1:eq1}
\nabla f(y)-\nabla f(x) = \nabla f(z(1))-\nabla f(z(0)) = \int_0^1 \frac{d}{dt}\nabla f(z(t))\,dt.
\end{equation}
By the chain rule, we have
\begin{align*}
\frac{d}{dt}\nabla f(z(t)) & =  \nabla^2 f(z(t)) z'(t)\\
 & =  \nabla^2 f(z(t)) (y-x) \text{ as }z'(t)=y-x
\end{align*}
which implies using \eqref{Lm1:eq1},
\begin{equation}\label{Lm1:eq2}
\nabla f(y)-\nabla f(x) = \int_0^1 \nabla^2 f(x+t(y-x))(y-x)\,dt.
\end{equation}
On subtracting $\nabla^2f(x)(y-x)$ from both sides in \eqref{Lm1:eq2}, we get
\begin{equation}\label{Lm1:eq3}
\nabla f(y)-\nabla f(x)-\nabla^2f(x)(y-x)= \int_0^1 \Big( \nabla^2f(x+t(y-x)) -\nabla^2f(x) \Big) (y-x)\,dt.
\end{equation}
Taking norms on both sides in \eqref{Lm1:eq3}, we have 
\begin{align*}
 \left\| \nabla f(y)-\nabla f(x)-\nabla^2f(x)(y-x) \right\| &\le \int_0^1 \left\lVert \nabla^2f(x+t(y-x)) -\nabla^2f(x) \right \rVert_2 \,\|y-x\|\,dt
\\
&\le L\|y-x\|^2 \int_0^1 t\,dt \text{ using Assumption }\ref{ass:lipschitz}\\
& = \frac{L}{2}\|y-x\|^2,
\end{align*}
which proves \eqref{eq:gradbound}.\\
Again, using the fundamental theorem of calculus, we have
\begin{align*}
& f(y) - f(x) -\langle\nabla f(x),y-x\rangle-\frac12\langle\nabla^2f(x)(y-x),y-x\rangle
\nonumber
\\
&= \int_0^1 \left\langle \nabla f(x+t(y-x)) - \nabla f(x) -
t\nabla^2f(x)(y-x), y-x \right\rangle dt
\label{eq:integral}
\end{align*}
which implies using
\eqref{eq:gradbound} that 
\begin{align*}
 &\left| f(y) - f(x) - \langle\nabla f(x),y-x\rangle - \frac12 \langle \nabla^2f(x)(y-x), y-x \rangle \right|\\
&\le \int_0^1 \frac{Lt^2}{2} \|y-x\|^3 \,dt
\\
&= \frac{L}{2}\|y-x\|^3 \int_0^1 t^2dt = \frac{L}{6}\|y-x\|^3,
\end{align*}
which establishes
\eqref{eq:taylorbound}.
\end{proof}

The estimate in \eqref{eq:taylorbound} naturally motivates the
construction of a cubic underestimator. Indeed, by subtracting a bound on the magnitude of the Taylor remainder from the second-order Taylor expansion, we obtain a valid underestimator of the objective function over the set $P$.

\begin{definition}\label{def:undover}
For an arbitrary but fixed expansion point $x\in P$ and a parameter $\Theta\geq L$, we define the 
{cubic underestimator} of $f$ over $P$ by 
\begin{equation}
m_x(y)=f(x)+\langle\nabla f(x),y-x\rangle+\frac12\langle\nabla^2f(x)(y-x),y-x\rangle-\frac{\Theta}{6}\|y-x\|^3 \ \forall y\in P.
\label{eq:under}
\end{equation}
Similarly, the associated cubic overestimator of $f$ for a fixed expansion point $x\in P$ and a parameter $\Theta\geq L$ is defined as
\begin{equation*}
M_x(y) = f(x) + \langle\nabla f(x),y-x\rangle + \frac12 \langle \nabla^2f(x)(y-x), y-x \rangle + \frac{\Theta}{6}\|y-x\|^3 \ \forall y\in P.
\label{eq:over}
\end{equation*}
\end{definition}
The following theorem establishes that the model
\eqref{eq:under} is an underestimator of $f$ over $P$.
\begin{theorem}\label{thm:under}
Let $x$ be any arbitrary but fixed point in $P$. Then, 
\[ m_x(y) \le f(y) \text{ and } M_{x}(y)\geq f(y) 
\text{ for all } y\in P.
\]
\end{theorem}
\begin{proof}
From Lemma \ref{lem:taylor}, for any $x,y\in P$ we have
\[ \left| R(x,y) \right| \le \frac{L}{6}\|y-x\|^3,
\]
where
$
R(x,y) = f(y) - f(x) - \langle\nabla f(x),y-x\rangle - \frac12 \langle \nabla^2f(x)(y-x), y-x \rangle
$.
Thus, we get 
\[ R(x,y) \ge -\frac{L}{6}\|y-x\|^3 \text { and } R(x,y) \le \frac{L}{6}\|y-x\|^3
\]
which implies using $\Theta\geq L$, for all $y$, we get  
\begin{equation*}
    f(x)+\langle\nabla f(x),y-x\rangle+\frac12\langle\nabla^2f(x)(y-x),y-x\rangle-\frac{\Theta}{6}\|y-x\|^3 \leq f(y)
    \end{equation*}
    \text{ and }
    \begin{equation*}
    f(x)+\langle\nabla f(x),y-x\rangle+\frac12\langle\nabla^2f(x)(y-x),y-x\rangle+\frac{\Theta}{6}\|y-x\|^3\geq f(y).
\end{equation*}
Hence, $m_{x}(y)\leq f(y)$ and $M_{x}(y)\geq f(y)$ for all $y\in P$ and the proof is complete.
\end{proof}

Although the choice $\Theta=L$ is sufficient to guarantee the validity of the cubic underestimator and overestimator given in Definition \ref{def:undover}, we introduce $\Theta$ as a general parameter satisfying $\Theta\geq L$. The underestimator $m_{x}$ can be nonconvex and finding its global minimum can be computationally expensive. Therefore, we derive conditions on $\Theta$ such that the underestimator $m_{x}$ is edge-concave on the box P. This allows us to exploit the vertex property of edge-concave functions and compute lower bounds within the branch-and-bound framework. Thus, $\Theta$ provides additional flexibility to ensure the edge-concavity of the underestimator while maintaining its validity as an underestimator.

Next, we prove that the cubic underestimator $m_x$ is edge-concave for suitable values of $\Theta$. 
 We start from the definition of an edge-concave function \cite{tardella2003polyhedral}.

\begin{definition}
    ([Remark 3.2 \cite{tardella2003polyhedral}]) A function $g: \mathbb{R}^{n}\rightarrow\mathbb{R}$ is said to be edge-concave on a box set $P$ if and only if $g$ is componentwise concave on the box set $P$.
\end{definition}

\begin{theorem} [Sufficient Condition for Edge Concavity]
\label{thm:edgeconcavity}
For any fixed $x\in P$, if $\Theta$ is chosen such that 
\begin{equation}
\Theta \ge \max\left\{ L,\; \frac{3}{5} \frac{\displaystyle\max_{1\le i\le n} \frac{\partial^2 f}{\partial x_i^2}(x)} {\omega(P)} \right\},
\label{eq:theta-condition}
\end{equation}
then the cubic underestimator $m_x$ is edge-concave on $P$.
\end{theorem}
\begin{proof}
Let $\phi :\mathbb{R}\rightarrow\mathbb{R}$ be a function given by $\phi (t)=m_x(y+te_{i})$, where  $e_i=(0,\ldots,1,\ldots,0)$ (\text{i.e.}, 1 at the $i$-th position) \text{ and is chosen such that } $y+te_i\in P$. We have
   $ \phi(t)=  f(x)+ \nabla f(x)^{\top }\textbf{u} + \frac{1}{2} \textbf{u}^{\top}\nabla^{2}f(x)\textbf{u}-\frac{\Theta}{6}\lVert y+te_{i}-x\rVert^{3}$,
   where $\textbf{u}=\left[y_{1}-x_{1}^{m}\;y_{2}-x_{2}^{m}\;
\cdots\;y_{i}+t-x_{i}\;\cdots\; y_{n}-x_{n}^{m}
\right]^{\top}$ with  $\nabla f(x)$ and $\nabla^{2} f(x)$ are given in \eqref{eq:Grad} and \eqref{eq:Hessian}, respectively.
    It can be easily seen that 
    \[
   \nabla f(x)^{\top }\textbf{u}=\sum\limits_{\substack{j=1 \\ j\neq i}}^{n} (y_j-x_{j})\frac{\partial f}{\partial x_{j}}(x)+ (y_i+t-x_{i})\frac{\partial f}{\partial x_{i}}(x)
    \] \text{and }
    \[ \textbf{u}^{\top}\nabla^{2}f(x)\textbf{u}=\sum\limits_{\substack{l=1, l\neq i  \\ j=1, j\neq i}}^{n} a_{lj}(y_{l}-x_{l})(y_{j}-x_{j})+2\sum\limits_{l=1, l\neq i}^{n}a_{li}(y_l-x_{l})(y_{i}+t-x_{i})+ a_{ii}(y_{i}+t-x_{i})^{2}.
    \]
    Thus, we have
    \begin{align*}
     \phi(t)
    = &f(x)+\sum\limits_{\substack{j=1 \\ j\neq i}}^{n} (y_j-x_{j})\frac{\partial f}{\partial x_{j}}(x)+ (y_i+t-x_{i})\frac{\partial f}{\partial x_{i}}(x)\\
    & + \frac{1}{2}\sum\limits_{\substack{l=1, l\neq i  \\ j=1, j\neq i}}^{n} a_{lj}(y_{l}-x_{l})(y_{j}-x_{j}) + \sum\limits_{l=1, l\neq i}^{n}a_{li}(y_l-x_{l})(y_{i}+t-x_{i})\\
     &  + \frac{1}{2} a_{ii}(y_{i}+t-x_{i})^{2}-\frac{\Theta}{6}\lVert y+te_{i}-x_{i}\rVert^{3}.    
    \end{align*}
To prove $m_{x}(y)$ is  edge-concave, we need to choose $\Theta$ such that $\phi$ is concave for each $i=1,2,\ldots, n$. To prove that $\phi$ is concave, it is sufficient to show that 
\begin{equation}\label{eq:edgecnv1}
\phi (t_{1})\leq \phi(t_2)+\phi'(t_2)(t_1-t_2) \text{ for any } t_{1}, t_{2} \text{ such that } y+t_1e_i\in P \text { and } y+t_2e_i\in P.
\end{equation}
Now we find the sufficient condition on $\Theta $ under which \eqref{eq:edgecnv1} holds. From \eqref{eq:edgecnv1}, we have
\begin{align}
    &f(x)+\sum\limits_{\substack{j=1 \\ j\neq i}}^{n} (y_j-x_{j})\frac{\partial f}{\partial x_{j}}(x)+ (y_i+t_1-x_{i})\frac{\partial f}{\partial x_{i}}(x)\nonumber\\
    &+\frac{1}{2}\sum\limits_{\substack{l=1, l\neq i  \\ j=1, j\neq i}}^{n} a_{lj}(y_{l}-x_{l})(y_{j}-x_{j})+\sum\limits_{l=1, l\neq i}^{n}a_{li}(y_{l}-x_{l})(y_{i}+t_{1}-x_{i})\nonumber\\
    &+ \frac{1}{2}a_{ii}(y_{i}+t_{1}-x_{i})^{2}-\frac{\Theta}{6}\lVert y+t_{1}e_{i}-x\rVert^{3}\nonumber\\
    &\leq f(x)+\sum\limits_{\substack{j=1 \\ j\neq i}}^{n} (y_j-x_{j})\frac{\partial f}{\partial x_{j}}(x)+ (y_i+t_2-x_{i})\frac{\partial f}{\partial x_{i}}(x)\nonumber\\
    &+\frac{1}{2}\sum\limits_{\substack{l=1, l\neq i  \\ j=1, j\neq i}}^{n} a_{lj}(y_{l}-x_{l})(y_{j}-x_{j})+\sum\limits_{l=1, l\neq i}^{n}a_{li}(y_{l}-x_{l})(y_{i}+t_{2}-x_{i})\nonumber\\
    &+\frac{1}{2} a_{ii}(y_{i}+t_{2}-x_{i})^{2}-\frac{\Theta}{6}\lVert y+t_{2}e_{i}-x\rVert^{3}+ \frac{\partial f}{\partial x_{j}}(x) (t_1-t_2) \nonumber\\
     &+ \sum\limits_{l=1, l\neq i}^{n}a_{li}(y_{l}-x_{l})(t_1-t_2)+ a_{ii} (y_i-x_{i})(t_{1}-t_{2})+a_{ii} t_{2} (t_1-t_2)\nonumber\\
     &~~~~~~~~~~~~~~~~~~~~~~~~~~~~~~~~~~~~~~~~~~~~~~~~~~~~-\frac{\Theta}{2}\lVert y+t_{2}e_{i}-x\rVert (y_{i}+t_{2}-x)(t_{1}-t_{2})\nonumber\\
    \implies&  (y_i+t_1-x_{i})\frac{\partial f}{\partial x_{i}}(x)+\sum\limits_{l=1, l\neq i}^{n}a_{li}(y_{l}-x_{l})(y_{i}+t_{1}-x_{i})+ \frac{1}{2}a_{ii}(y_{i}+t_{1}-x_{i})^{2}\nonumber\\
    &~~~~~~~~~~~~~~~~~~~~~~~~~~~~~~~~~~~~~~~~~~~~~~~~~~~~~~~~~~~~~~~~~~~~~~~~~~~~~~~~~-\frac{\Theta}{6}\lVert y+t_{1}e_{i}-x\rVert^{3}\nonumber\\
    & \leq (y_i+t_2-x_{i})\frac{\partial f}{\partial x_{i}}(x)+\sum\limits_{l=1, l\neq i}^{n}a_{li}(y_{l}-x_{l})(y_{i}+t_{2}-x_{i})+ \frac{1}{2} a_{ii}(y_{i}+t_{2}-x_{i})^{2}\nonumber\\
     &-\frac{\Theta}{6}\lVert y+t_{2}e_{i}-x\rVert^{3} + \frac{\partial f}{\partial x_{j}}(x) (t_1-t_2) + \sum\limits_{l=1, l\neq i}^{n}a_{li}(y_{l}-x_{l})(t_1-t_2)\nonumber\\
     &+ a_{ii} (y_i-x_{i})(t_{1}-t_{2})+a_{ii} t_{2} (t_1-t_2)-
     \frac{\Theta}{2}\lVert y+t_{2}e_{i}-x\rVert (y_{i}+t_{2}-x_{i})(t_{1}-t_{2})\nonumber\\
    \implies & (y_i-x_{i})\frac{\partial f}{\partial x_{i}}(x)+t_1\frac{\partial f}{\partial x_{i}}(x)+\sum\limits_{l=1, l\neq i}^{n}a_{li}(y_{l}-x_{l})(y_{i}-x_{i})+\sum\limits_{l=1, l\neq i}^{n}a_{li}(y_{l}-x_{l})t_1\nonumber\\
    &+\frac{1}{2}a_{ii}((y_i-x_{i})^{2}+2(y_{i}-x_{i})t_1+t_{1}^{2})-\frac{\Theta}{6}\lVert y+t_{1}e_{i}-x\rVert^{3}\nonumber\\
    & \leq  (y_i-x_{i})\frac{\partial f}{\partial x_{i}}(x)+t_2\frac{\partial f}{\partial x_{i}}(x)+\sum\limits_{l=1, l\neq i}^{n}a_{li}(y_{l}-x_{l})(y_{i}-x_{i})\nonumber\\
    &+\sum\limits_{l=1, l\neq i}^{n}a_{li}(y_{l}-x_{l})t_2+\frac{1}{2}a_{ii}((y_i-x_{i})^{2}+2(y_{i}-x_{i})t_2+t_{2}^{2})-\frac{\Theta}{6}\lVert y+t_{2}e_{i}-x\rVert^{3}\nonumber\\
    & + \frac{\partial f}{\partial x_{j}}(x) (t_1-t_2) + \sum\limits_{l=1, l\neq i}^{n}a_{li}(y_{l}-x_{l})(t_1-t_2) + a_{ii} (y_i-x_{i})(t_{1}-t_{2})\nonumber\\
     &+a_{ii} t_{2} (t_1-t_2)- \frac{\Theta}{2}\lVert y+t_{2}e_{i}-x\rVert (y_{i}+t_{2}-x_{i})(t_{1}-t_{2})\nonumber\\
     \implies &\frac{1}{2}a_{ii}t_{1}^{2}-\frac{\Theta}{6}\lVert y+t_{1}e_{i}-x\rVert^{3} \leq \frac{1}{2}a_{ii}t_{2}^{2}-\frac{\Theta}{6}\lVert y+t_{2}e_{i}-x\rVert^{3}+ a_{ii} t_{2} (t_1-t_2)\nonumber\\
     &~~~~~~~~~~~~~~~~~~~~~~~~~~~~~~~~~~~~~~~~~- \frac{\Theta}{2}\lVert y+t_{2}e_{i}-x\rVert (y_{i}+t_{2}-x_{i})(t_{1}-t_{2})\nonumber\\
     \implies & \frac{1}{2} a_{ii}(t_{1}^{2}-t_{2}^{2}-2t_{2}(t_{1}-t_{2}))\leq \frac{\Theta}{6}\lVert y+t_{1}e_{i}-x\rVert^{3}-\frac{\Theta}{6}\lVert y+t_{2}e_{i}-x\rVert^{3}\nonumber\\
     &~~~~~~~~~~~~~~~~~~~~~~~~~~~~~~~~~~~~~~~~~- \frac{\Theta}{2}\lVert y+t_{2}e_{i}-x\rVert (y_{i}+t_{2}-x_{i})(t_{1}-t_{2})\nonumber\\
     \implies & \frac{1}{2}a_{ii} (t_{1}-t_{2})^{2}\leq \frac{\Theta}{6}\lVert y+t_{1}e_{i}-x\rVert^{3}+\frac{\Theta}{6}\lVert y+t_{2}e_{i}-x\rVert^{3}\nonumber\\
    &~~~~~~~~~~~~~~~~~~~~~~~~~~~~~~~~~~~~~~~~~+ \frac{\Theta}{2}\lVert y+t_{2}e_{i}-x\rVert \vert y_{i}+t_{2}-x_{i}\rvert \lvert t_{1}-t_{2}\rvert.\label{ec1}
\end{align}
Note that $t_{1}$ and $t_{2}$ are chosen such that $y+t_{1}e_{i}\in P$ and $y+t_{2}e_{i}\in P$. Therefore, we have 
\begin{equation}\label{bound}
\lVert y+t_{1}e_{i}-x\rVert \leq \omega (P), ~  (t_{1}-t_{2})^{2}\leq \omega (P)^{2} \text{ and } \lvert t_{1}-t_{2}\rvert \leq \omega (P). 
\end{equation}
From \eqref{ec1} and \eqref{bound}, we get 
\begin{align}\label{ec2}
    \frac{1}{2}a_{ii} (t_{1}-t_{2})^{2}&\leq \frac{\Theta}{6} \omega (P)^{3}+\frac{\Theta}{6} \omega (P)^{3}+ \frac{\Theta}{2}\omega(P)^{3}\nonumber\\
    &= \frac{5\Theta}{6}\omega (P)^{3}.
\end{align}
Since $\Theta\geq L\geq 0$, therefore \eqref{ec2} is satisfied for $a_{ii}\leq0$. Now we consider $a_{ii}\geq0$. Our aim is to find a $\Theta$ such that \eqref{ec2} holds. Thus, if we find the $\Theta$ such that it satisfies
\begin{align*}
    \frac{1}{2}a_{ii} \omega(P)^{2}\leq \frac{5\Theta}{6}\omega (P)^{3},
\end{align*}
then it satisfies \eqref{ec2} automatically as $(t_{1}-t_{2})^{2}\leq \omega (P)^{2}$ and $a_{ii}\geq0$. Then, 
\begin{align*}
    \frac{1}{2}a_{ii} \omega(P)^{2}\leq \frac{5\Theta}{6}  \omega (P)^{3}
\end{align*}
which implies that 
\[
\Theta \geq \frac{3}{5} \frac{a_{ii}}{\omega (P)}.
\]
Hence, if we choose $\Theta$ such that $\Theta=\max\left\{\frac{3}{5} \frac{a_{ii}}{\omega (P)},L\right\}$ where $a_{ii}= \frac{\partial^2 f}{\partial x_i^2}(x)$, then $\phi(t)$ is concave along the $i$-th coordinate. Subsequently, if we take $\Theta$ such that  $\Theta=\max\left\{\frac{3}{5} \frac{\textbf{a}}{\omega (P)},L\right\}$, where $\textbf{a}=\max {a_{ii}}$ for each $i=1,2,\ldots,n$, then $\phi$ is componentwise concave, which implies that $m_{x}$ is edge-concave.
\end{proof}
 For the rest of the analysis in the paper, we use $\Theta$ such that it satisfies the condition \eqref{eq:theta-condition}.


\begin{remark}
A key property of edge-concave functions is that their minimum over a box is attained at a vertex. In particular, if $g$ is edge-concave on
$P$, then
\begin{equation*}
    \min_{x\in P}g(x) =\min_{v\in V(P)}g(v).
\end{equation*}
Therefore, an edge-concave underestimator provides a lower bound that can be computed by evaluating the underestimator at the vertices of the box \cite{tardella1988class}.
\end{remark}

\section{Convergence analysis}\label{sec:convergence}

In this section, we analyze the estimation error associated with the proposed cubic estimator scheme. We first establish bounds on the errors of the cubic underestimator and overestimator. We then use these error bounds to establish pointwise and Hausdorff convergence of order three.

For each $Y\subseteq P$, let $x_{Y}$ be a fixed point in $Y$.
Then, the family
\begin{equation}\label{eq:schemes}
(m_{x_Y},M_{x_Y})_{Y\subseteq P}
\end{equation}
is called the scheme of cubic estimators, where $m_{x_{Y}}$ and $M_{x_{Y}}$ denote the underestimator and overestimator of $f$ over $Y$, respectively.

 The pointwise convergence order shows how quickly the estimators become close to the objective function as the size of the box decreases. The Hausdorff convergence order measures the rate at which the interval enclosure generated by the estimators approaches the exact range (the set of all function values attained by the objective function over the considered box) of the objective function.
 In \cite{bompadre2012convergence}, both the pointwise and Hausdorff convergence orders were proven for a scheme of estimators in which the underestimator is convex and the overestimator is concave. However, the proposed scheme \eqref{eq:schemes} does not, in general, satisfy these convexity and concavity properties. Therefore, the convergence theory developed for convex underestimator/concave overestimator schemes in \cite{bompadre2012convergence} is not directly applicable.
Instead, we follow the generalized framework introduced in
\cite{bompadre2013convergence} for parameterized schemes of estimators to analyse the pointwise and Hausdorff convergence orders of the proposed scheme \eqref{eq:schemes}.

\subsection{Pointwise convergence of the scheme of estimators}
 \begin{definition}[Pointwise convergence order \cite{bompadre2013convergence}]\label{ptcgs}
The scheme $(m_{x_Y},M_{x_Y})_{Y\subseteq P}$ of $f$ is said to have {pointwise convergence of order $\gamma$} on $P$ if there exists a constant $\tau>0$ such that
\[
\sup_{y\in Y} |f(y)-m_{x_Y}(y)| \le \tau\,w(Y)^\gamma, \text{ and } \sup_{y\in Y} |M_{x_Y}(y)-f(y)| \le \tau\,w(Y)^\gamma,
\]
 for every $Y\subseteq P$ and for any fixed $x_{Y}\in Y$.
\end{definition}
The next result shows that the estimation error of both the underestimator and the overestimator decreases cubically with the distance from the expansion point.
\begin{lemma}
\label{thm:pointwise}
 Let $x\in P$ be any fixed expansion point. Then, for every $y\in P$,
\begin{equation*}
0 \le f(y)-m_x(y) \le \left(\frac{L+\Theta}{6}\right)\|y-x\|^3,
\label{eq:undererror} \text{ and } 0\le M_x(y)-f(y) \le \left(\frac{L+\Theta}{6}\right)\|y-x\|^3,
\end{equation*}
Therefore, as $y\rightarrow x$, we have
\[ f(y)-m_x(y) = \text{ O}(\|y-x\|^3), \text{ and } M_x(y)-f(y) = \text{O}(\|y-x\|^3).
\]
\end{lemma}
\begin{proof}
From Lemma \ref{lem:taylor}, we have
\begin{equation}\label{eq:boundRxy}
-\frac{L}{6}\|y-x\|^3 \le R(x,y)\le\frac{L}{6}\|y-x\|^3,
\end{equation}
where $R(x,y)=f(y)-f(x)-\langle\nabla f(x),y-x\rangle - \frac12 \langle \nabla^2f(x)(y-x),y-x \rangle $.
From Definition \ref{def:undover}, we have
\[
m_x(y)= f(x) + \langle\nabla f(x),y-x\rangle + \frac12 \langle \nabla^2f(x)(y-x), y-x\rangle -\frac{\Theta}{6}\|y-x\|^3,
\]
which implies along with \eqref{eq:boundRxy} that
\[
f(y)-m_x(y) = R(x,y) +\frac{\Theta}{6}\|y-x\|^3 \leq \left(\frac{L+\Theta}{6}\right)\|y-x\|^3.
\]
Therefore, we get
\begin{equation}\label{eq:mxbound}
\lvert f(y)-m_x(y)\rvert\le\left(\frac{L+\Theta}{6}\right)\|y-x\|^3.
\end{equation}
Similarly using Definition \ref{def:undover} and \eqref{eq:boundRxy}, we have
\[
M_x(y)-f(y) =\frac{\Theta}{6}\|y-x\|^3 -R(x,y)\leq \left(\frac{L+\Theta}{6}\right) \|y-x\|^3
\]
which yields
\begin{equation}\label{eq:MXbound}
\lvert M_x(y)-f(y)\rvert \le \left(\frac{L+\Theta}{6}\right) \|y-x\|^3.
\end{equation}
Thus,  \eqref{eq:mxbound} and \eqref{eq:MXbound} imply, as $y\rightarrow x$,
\[
f(y)-m_x(y)=\text{O}(\|y-x\|^3),
\text{ and }
M_x(y)-f(y)=\text{O}(\|y-x\|^3)
\]
and the proof is complete.
\end{proof}

 The previous result immediately yields the following pointwise convergence theorem.

\begin{theorem}\label{thmptcgs}
For any $Y\subseteq P$ and a fixed point $x_{Y}\in Y$,
 the scheme of cubic estimators $(m_{x_Y},M_{x_Y})_{Y\subseteq P}$ has pointwise convergence of order three.
\end{theorem}
\begin{proof}
From Lemma \ref{thm:pointwise} for any $y\in Y$, we have
\begin{equation}\label{eq:mxMxbound}
|f(y)-m_{x_{Y}}(y)| \le \left(\frac{L+\Theta}{6}\right) \|y-x_{Y}\|^3 \text{ and } |M_{x_{Y}}(y)-f(y)| \le\left(\frac{L+\Theta}{6}\right) \|y-x_{Y}\|^3
\end{equation}
Since $x_{Y}\in Y$, therefore
\[
\|y-x_{Y}\| \le \omega(Y), \text{ for all }  y\in Y,
\]
which implies using \eqref{eq:mxMxbound} that
\begin{equation}\label{eq:mxMxboundcgs}
|f(y)-m_{x_{Y}}(y)| \le \left(\frac{L+\Theta}{6}\right)\omega(Y)^3 \text{ and } |M_{x_{Y}}-f(y)| \le \left(\frac{L+\Theta}{6}\right)\omega(Y)^3.
\end{equation}
Taking the supremum over $Y$ in \eqref{eq:mxMxboundcgs}, we get 
\begin{equation}\label{eq:mxMxboundcgs1}
\sup_{y\in Y} |f(y)-m_{x_Y}(y)| \le \left(\frac{L+\Theta}{6}\right)\omega(Y)^3
\text{ and }
\sup_{y\in Y} |M_{x_{Y}}(y)-f(y)| \le \left(\frac{L+\Theta}{6}\right)\omega(Y)^3.
\end{equation}
Hence, Definition \ref{ptcgs} holds with
\[
\tau=\frac{L+\Theta}{6} \text{ and } \gamma=3.
\]
Thus, the scheme of cubic estimators $(m_{x_Y},M_{x_Y})_{Y\subseteq P}$ has pointwise convergence of order three.
\end{proof}

\begin{remark} \label{rem:midpoint}
The bound in Theorem \ref{thmptcgs} can be improved by choosing $x_Y$ as the midpoint of $Y$. In this case,
\begin{equation}\label{eq:rmrkcgs}
\|y-x_Y\|\leq \frac{\omega(Y)}{2},
\text{ for all }y\in Y.
\end{equation}
Therefore, using \eqref{eq:mxMxboundcgs} and \eqref{eq:rmrkcgs}
\begin{equation}\label{eq:rmrkptcgs}
\sup_{y\in Y}|f(y)-m_{x_Y}(y)| \leq \left(\frac{L+\Theta}{48}\right)\omega(Y)^3 \text{ and } \sup_{y\in Y}|M_{x_Y}(y)-f(y)| \leq \left(\frac{L+\Theta}{48}\right)\omega(Y)^3.
\end{equation}
Thus, when the midpoint of $Y$ is used as the expansion point, the constant $\tau$ in the pointwise convergence bound is reduced from $\frac{L+\Theta}{6}$ to $\frac{L+\Theta}{48}$. Since the latter provides a tighter error bound, the midpoint is used as the expansion point in our implementation.
\end{remark}

\subsection{Hausdorff convergence of the scheme of estimators}

We establish that the proposed scheme of cubic estimators has Hausdorff convergence of order three. To this end, we use the approximation error bounds obtained in the previous subsection. We first introduce the inclusion mapping associated with the cubic estimator scheme and recall the Hausdorff metric for compact sets. We then define Hausdorff convergence order and prove that the proposed scheme attains order three.

For each $Y\subseteq P$, let $
\overline f(Y) = \left[ \min\limits_{y\in Y}f(y), \; \max\limits_{y\in Y}f(y) \right]$
denote the image of $Y$ under $f$. Note that the cubic estimator scheme defines the interval-valued map such that 
$H(Y) =  \left[ \min\limits_{y\in Y}m_{x_Y}(y), \; \max\limits_{y\in Y}M_{x_Y}(y) \right].$
Since $m_{x_Y}$ is an underestimator and $M_{x_Y}$ is an overestimator of $f$ on $Y$, $\overline f(Y) \subseteq H(Y)$.
Thus, $H(Y)$ provides an interval enclosure of the range of $f$ over $Y$.
\begin{definition}[Hausdorff metric \cite{bompadre2012convergence}] \label{hauscgsdef}
Let $A,B\subseteq\mathbb{R}$ be two nonempty compact sets.
The Hausdorff metric between $A$ and $B$ is defined by
\[
q(A,B)=\max\left\{\sup_{a\in A}\inf_{b\in B}|a-b|,
\;\sup_{b\in B}\inf_{a\in A}|a-b|\right\}.
\]
In particular, if $A=[a_1,a_2]$ and $B=[b_1,b_2]$ are compact intervals, then
\[
q(A,B)=\max\left\{|a_1-b_1|,|a_2-b_2|\right\}.
\]
\end{definition}

\begin{definition}[Hausdorff convergence order \cite{bompadre2013convergence}]\label{hauscgs}
The scheme of estimators $(m_{x_Y},M_{x_Y})_{Y\subseteq P}$ is said to have {Hausdorff convergence of order} $\beta>0$
on $P$ if there exists a constant $\tau>0$ such that
\[
q(\overline f(Y),H(Y)) \le \tau\,w(Y)^\beta,
\]
for every $Y\subseteq P$ and for any fixed $x_{Y}\in Y$.
\end{definition}

\begin{theorem}\label{thm:hausdorff}
For any $Y\subseteq P$ and a fixed point $x_{Y}\in Y$, the  proposed scheme of cubic estimators $(m_{x_Y},M_{x_Y})_{Y\subseteq P}$ has Hausdorff convergence of order three.
\end{theorem}
\begin{proof}
Since $m_{x_Y}(y)\leq f(y)$ for every $y\in Y$, we have
$\min\limits_{y\in Y}m_{x_Y}(y) \leq \min\limits_{y\in Y}f(y)$. Let $y^\ast\in Y$ be such that $m_{x_Y}(y^\ast)=\min\limits_{y\in Y}m_{x_Y}(y)$. Then, we have
\begin{align}
\label{eq:mxboundhsdrf}
0\leq \min_{y\in Y}f(y)-\min_{y\in Y}m_{x_Y}(y) &= \min_{y\in Y}f(y)-m_{x_Y}(y^\ast) \nonumber\\
&\leq f(y^\ast)-m_{x_Y}(y^\ast) \nonumber\\ 
&\leq \sup_{y\in Y}\left(f(y)-m_{x_Y}(y)\right) \nonumber\\
&\leq \left(\frac{L+\Theta}{6}\right)\omega(Y)^3,
\end{align}
where the last inequality follows from \eqref{eq:mxMxboundcgs1}.
Similarly, from $\max\limits_{y\in Y}f(y) \le \max\limits_{y\in Y}M_{x_Y}(y),$ we have
\begin{align}\label{eq:MxboundHsdrf}
0\leq \max_{y\in Y}M_{x_Y}(y)- \max_{y\in Y}f(y) &\le \sup_{y\in Y} \left(M_{x_Y}(y)-f(y) \right)\nonumber\\
&\le \left(\frac{L+\Theta}{6}\right)\omega(Y)^3 \text{ using }\eqref{eq:mxMxboundcgs1}. 
\end{align}
Using Definition \ref{hauscgsdef}, we have 
\[q(\overline f(Y),H(Y)) = \max \left\{ \min_{y\in Y}f(y)-\min_{y\in Y}m_{x_Y}(y), \; \max_{y\in Y}M_{x_Y}(y)-\max_{y\in Y}f(y)
\right\}
\]
which implies, using \eqref{eq:mxboundhsdrf} and \eqref{eq:MxboundHsdrf}, that
\[
q(\overline{f}(Y),H(Y)) \le \left(\frac{L+\Theta}{6}\right)\omega(Y)^3.
\]
Hence, Definition \ref{hauscgs} holds with
\[
\tau=\left(\frac{L+\Theta}{6}\right) 
\text{  and }
\beta=3
\]
Thus, the scheme of cubic estimators $(m_{x_Y},M_{x_Y})_{Y\subseteq P}$ has the Hausdorff convergence of order three.
\end{proof}

\begin{remark}
\label{rem:hausdorff_midpoint}
If $x_Y$ is chosen as the midpoint of $Y$, then using \eqref{eq:rmrkptcgs}, \eqref{eq:mxboundhsdrf} and \eqref{eq:MxboundHsdrf}, we get
\[
q\left(\overline{f}(Y),H(Y)\right) \leq \left(\frac{L+\Theta}{48}\right)\omega(Y)^3.
\]
Thus, the midpoint choice improves the convergence constant $\tau$ from $\frac{L+\Theta}{6}$ to $\frac{L+\Theta}{48}$, while the order of Hausdorff convergence remains three.
\end{remark}

The preceding convergence analysis describes the estimation error in terms of the diameter of an arbitrary subbox. We next choose the midpoint of the initial box $P$ as the expansion point and derive an explicit bound on the maximum separation between the objective function and the proposed cubic underestimator over $P$. The resulting bound quantifies the tightness of the underestimator in terms of $L$ and the width of the box.

\begin{proposition}
\label{prop:max_separation_cubic}
Let $x^M$ denote the midpoint of $P$. The maximum separation between $f$ and the proposed cubic underestimator satisfies
\begin{equation*}\label{eq:max_separation_cubic}
d_P^{\mathrm{cub}}= \max_{y\in P} \{f(y)-m_{x^M}(y)\} \leq \frac{L+\Theta}{48} \left(\sum_{i=1}^n (x_i^U-x_i^L)^2
\right)^{3/2}.
\end{equation*}
\end{proposition}

\begin{proof}
From \eqref{eq:mxMxbound}, we have
\begin{equation}\label{eq:propdist}
f(y)-m_{x^M}(y) \leq \left(\frac{L+\Theta}{6}\right)\|y-x^M\|^3.
\end{equation}
Taking the maximum over $y\in P$ in \eqref{eq:propdist} gives
\[
d_P^{\mathrm{cub}} \leq \left(\frac{L+\Theta}{6}\right) \max_{y\in P}\|y-x^M\|^3.
\]
which implies using 
$ \max_{y\in P}\|y-x^M\| = \frac{1}{2} \left(\sum_{i=1}^n (x_i^U-x_i^L)^2 \right)^{1/2}$, that
\[
d_P^{\mathrm{cub}}
\leq
\frac{L+\Theta}{48}
\left(
\sum_{i=1}^n (x_i^U-x_i^L)^2
\right)^{3/2}
\]
and the proof is complete.
\end{proof}

\begin{remark}
   Note that setting $\Theta=L$ in Theorems \ref{thmptcgs} and \ref{thm:hausdorff} still yields pointwise and Hausdorff convergence of order three, respectively. The convergence analysis of the proposed estimator scheme relies on the validity of the  cubic underestimator and overestimator over the box $P$. It does not require the cubic underestimator to satisfy the edge-concavity property. Edge-concavity is only used separately in the branch-and-bound framework to compute lower bounds by evaluating the underestimator at the vertices of each box.
\end{remark}

\section{A spatial branch-and-bound algorithm using CECU}\label{sec:bnb}
We incorporate CECU into a deterministic branch-and-bound framework for solving nonconvex global optimization problems over box-constrained domains. At each node of the branch-and-bound tree, CECU provides a valid lower bound over the corresponding box. The resulting lower bound can be used for node selection and pruning, while a local optimization procedure is used to obtain feasible points and update the global upper bound. The overall branch-and-bound procedure is summarized in Algorithm \ref{alg:bnb}. The computation of the lower bound for each subbox is described in Algorithm \ref{alg:lowerbound}.

\begin{algorithm}[H]
\caption{Branch-and-Bound Using CECU}
\label{alg:bnb}
\begin{algorithmic}[1]

\State Choose a tolerance $\varepsilon>0$, $P_0=[x^{L}, x^{U}]$ and a maximum number of
iterations $k_{\max}$.
\State Set the iteration counter $k=0$.
\State Initialize the active list $\mathcal{A}=\{P_0\},$

\State Apply a local optimization solver to $f$ over $P_0$ and obtain a feasible point $x_{\mathrm{inc}}$.
\State Set $ UB=f(x_{\mathrm{inc}}).$

\State Compute the lower bound $LB(P_0)$ using Algorithm \ref{alg:lowerbound}.
\State Set $LB=LB(P_0)$.

\While{$UB-LB>\varepsilon$
\textbf{and}
$\displaystyle
\frac{UB-LB}{\max\{1,LB\}}>\varepsilon$
\textbf{and}
$k\leq k_{\max}$}

    \State Select the active subbox $P\in\mathcal{A}$ having the smallest lower bound $LB(P)$.

    \State Bisect $P$ along its longest edge to obtain two child node $P_L$ and $P_R$.

    \State Remove $P$ from the active list.

    \For{$Q\in\{P_L,P_R\}$}

        \State Compute the lower bound $LB(Q)$ using
        Algorithm \ref{alg:lowerbound}.

        \State Apply a local optimization solver to $f$ over $Q$ and get a feasible point $x_Q$.

        \If{$f(x_Q)<UB$}
            \State $x_{\mathrm{inc}}\leftarrow x_Q$.
            \State $UB\leftarrow f(x_Q)$.
        \EndIf

        \If{$LB(Q)<UB$}
            \State Insert $Q$ into the active list.
        \EndIf

    \EndFor

    \If{$\mathcal{A}\neq\emptyset$}
        \State Update the global lower bound $LB=\min\limits_{Q\in\mathcal{A}}LB(Q)$.
    \Else
        \State \textbf{break}.
    \EndIf

    \State $k\leftarrow k+1$.

\EndWhile

\State \textbf{Return} $x_{\mathrm{inc}}$ and $UB$.

\end{algorithmic}
\end{algorithm}

For each active subbox, a valid lower bound is computed using the proposed CECU. The midpoint of the subbox is used as the expansion point, and  CECU is constructed using a Hessian Lipschitz constant and the parameter $\Theta$. Since CECU is edge-concave (see Remark 1), its minimum over the subbox is attained at a vertex. The complete procedure is given below in Algorithm \ref{alg:lowerbound}.

\begin{algorithm}[H]
\caption{Lower Bound Computation Using CECU}
\label{alg:lowerbound}
\begin{algorithmic}[1]

\Require Box set $P=[x^{L}, x^{U}]\subset\mathbb{R}^n$
\Ensure Lower bound $LB(P)$
\State Compute the midpoint $x^M=\frac{x^L+x^U}{2}$.

\State Evaluate
$f(x^M),  \nabla f(x^M), \text{ and } \nabla^2 f(x^M)$.

\State Compute a valid Hessian Lipschitz constant $L_P$ on $P$ (details in subsection \ref{subsec:lipschitz_estimation}).

\State Calculate
$
\Theta =
\max\left\{L_P,\frac{3}{5\omega(P)}\max_{1\leq i\leq n}\frac{\partial^2 f}{\partial x_i^2}(x^M)\right\}.
$

\State Construct CECU
\[
m_{x^M}(y)= f(x^M) +\nabla f(x^M)^T(y-x^M) +\frac12(y-x^M)^{\top} \nabla^2f(x^M)(y-x^M) -\frac{\Theta}{6}\|y-x^M\|^3.
\]

\For{each vertex $v\in V(P)$}
    \State Evaluate $m_{x^M}(v)$.
\EndFor

\State Set $LB(P)=\min\limits_{v\in V(P)}m_{x^M}(v)$.

\State \Return $LB(P)$.
\end{algorithmic}
\end{algorithm}

\subsection{Estimation of the Hessian Lipschitz Constant}
\label{subsec:lipschitz_estimation}
Note that the exact Hessian Lipschitz constant on $P$ is defined as
\[
L = \sup_{\substack{x,y\in P\\x\neq y}}\frac{\rVert\nabla^2f(x)-\nabla^2f(y)\rVert_2}{\lVert x-y\rVert}.
\]
In general, computing this exact constant requires solving a global optimization problem and can be computationally demanding. Therefore, in the numerical implementation, we use a computationally tractable upper bound $\hat L$ of the exact Lipschitz constant $L$ based on the Frobenius norm of the third-order partial derivatives of $f$. The Hessian Lipschitz constant $L_P$ used in
Algorithm~\ref{alg:lowerbound} is estimated using interval arithmetic.
Let $D^3f(z)$ denote the third derivative tensor of $f$ at
$z\in P$. For each $i,j,k=1,2,\ldots,n$, interval arithmetic is used to
obtain an upper bound $U_{ijk}$ satisfying
\[
U_{ijk}\geq
\sup_{z\in P}
\left|
\frac{\partial^3 f(z)}
{\partial x_i\partial x_j\partial x_k}
\right|.
\]
We then define
\begin{equation}
\label{eq:Lipschitz_upper_bound}
\hat{L}= \bigg( \sum_{i,j,k=1}^{n} U_{ijk}^2\bigg)^{1/2}.
\end{equation}
For any $x,y\in P$,
\[
\nabla^2 f(x)-\nabla^2 f(y) = \int_0^1 D^3f\bigl(y+t(x-y)\bigr)[x-y]\,dt,
\]
which implies that
\[
\begin{aligned}
\|\nabla^2 f(x)-\nabla^2 f(y)\|_2
&\leq \sup_{z\in P}\|D^3f(z)\|_{F}\|x-y\|\leq \hat{L}\|x-y\|.
\end{aligned}
\]
Thus, $\hat{L}$ obtained from \eqref{eq:Lipschitz_upper_bound} is a computationally tractable upper bound for the Hessian Lipschitz constant $L$ on $P$. 
We consider an illustrative example to show the methodology to construct the underestimator.

\begin{example}
\label{ex:cubic_underestimator}
Consider the nonconvex function $f:\mathbb{R}^{2}\rightarrow\mathbb{R}$ on a box set $P=[-1,1]\times [-1,1]$ in $\mathbb{R}^{n}$ such that  $f(x_1,x_2)=x_1^2x_2.$
Let us calculate the proposed cubic underestimator for the given function $f$ on $P$ by taking the midpoint $x^{M}=(0,0)$ of $P$ as an expansion point. For this, we need $\nabla f(x^M),\nabla^2f(x^M)$, and Lipschitz Hessian constant $L$ of $f$ on $P$. The gradient and Hessian of $f$ are given by
\[
\nabla f(x_1,x_2)= \begin{pmatrix} 2x_1x_2\\ x_1^2 \end{pmatrix} \text{  and }\nabla^2 f(x_1,x_2)
=\begin{pmatrix}
2x_2 & 2x_1\\
2x_1 & 0
\end{pmatrix},
\]
respectively. Thus, we have $ \nabla f(x^M)=\begin{pmatrix} 0\\ 0 \end{pmatrix}$ and
$\nabla^2f(x^M)=\begin{pmatrix} 0 & 0\\ 0 & 0 \end{pmatrix}$. Computing the exact Hessian Lipschitz constant $L$ on $P$ can be computationally expensive. Therefore, instead of calculating the exact Hessian Lipschitz constant, we use an upper bound $\widehat{L}$ satisfying $L\leq \widehat{L}$. This upper bound is obtained using the Frobenius norm of the third derivative tensor of $f$. For the considered function, we have
\[
\begin{aligned}
\|\nabla^3f(z)\|_F &= \left( \sum_{i,j,k=1}^{2} \left\vert
\frac{\partial^3f(z)} {\partial x_i\partial x_j\partial x_k} \right\rvert^2 \right)^{1/2}\\
&= \sqrt{2^2+2^2+2^2}= 2\sqrt{3}.
\end{aligned}
\]
Since the third order derivatives are constant on $P$, we have
\[
\sup_{z\in P}\|\nabla^3f(z)\|_F=2\sqrt{3}.
\]
Thus, we get
\begin{align*}
\|\nabla^2f(x)-\nabla^2f(y)\|_2
&\leq \int_0^1 \left\| \nabla^3f\bigl(y+t(x-y)\bigr) \right\|_F \|x-y\|\,dt= 2\sqrt{3}\,\|x-y\|.
\end{align*}
Therefore, we have  $L\leq 2\sqrt{3}=\hat{L}$. Hence, using $f(x^{M})=0, \nabla f(x^M)=\begin{pmatrix} 0\\ 0 \end{pmatrix}$,
$\nabla^2f(x^M)=\begin{pmatrix} 0 & 0\\ 0 & 0 \end{pmatrix}$, and $\Theta=2\sqrt{3}$ in Definition \ref{def:undover}, the cubic underestimator $m_{x^{M}}$ for $f$ on $P$ is given by
\[
m_{x^M}(y_1,y_2)= -\tfrac{\sqrt{3}}{3} \left(y_1^2+y_2^2\right)^{3/2} \text{ where } (y_{1}, y_{2})\in P.
\]
We plot the graphs of $f$ and its underestimator $m_{x^M}$ on $P$ considered in Example \ref{ex:cubic_underestimator} in Figure \ref{fig:objective_cecu}. The cyan color plot represents the graph of the function $f$, while the pink color plot represents the graph of the underestimator $m_{x^M}$ constructed at the midpoint $x^M=(0,0)$ of $P$. As we can see from the figure, the plot corresponding to $m_{x^M}$ lies below the graph of $f$ on $P$, which implies that $m_{x^M}(x_1,x_2)\leq f(x_1,x_2)$ for all $(x_1,x_2)\in P$. 
\end{example}
\begin{figure}[H]
\centering
\includegraphics[width=0.48\textwidth]{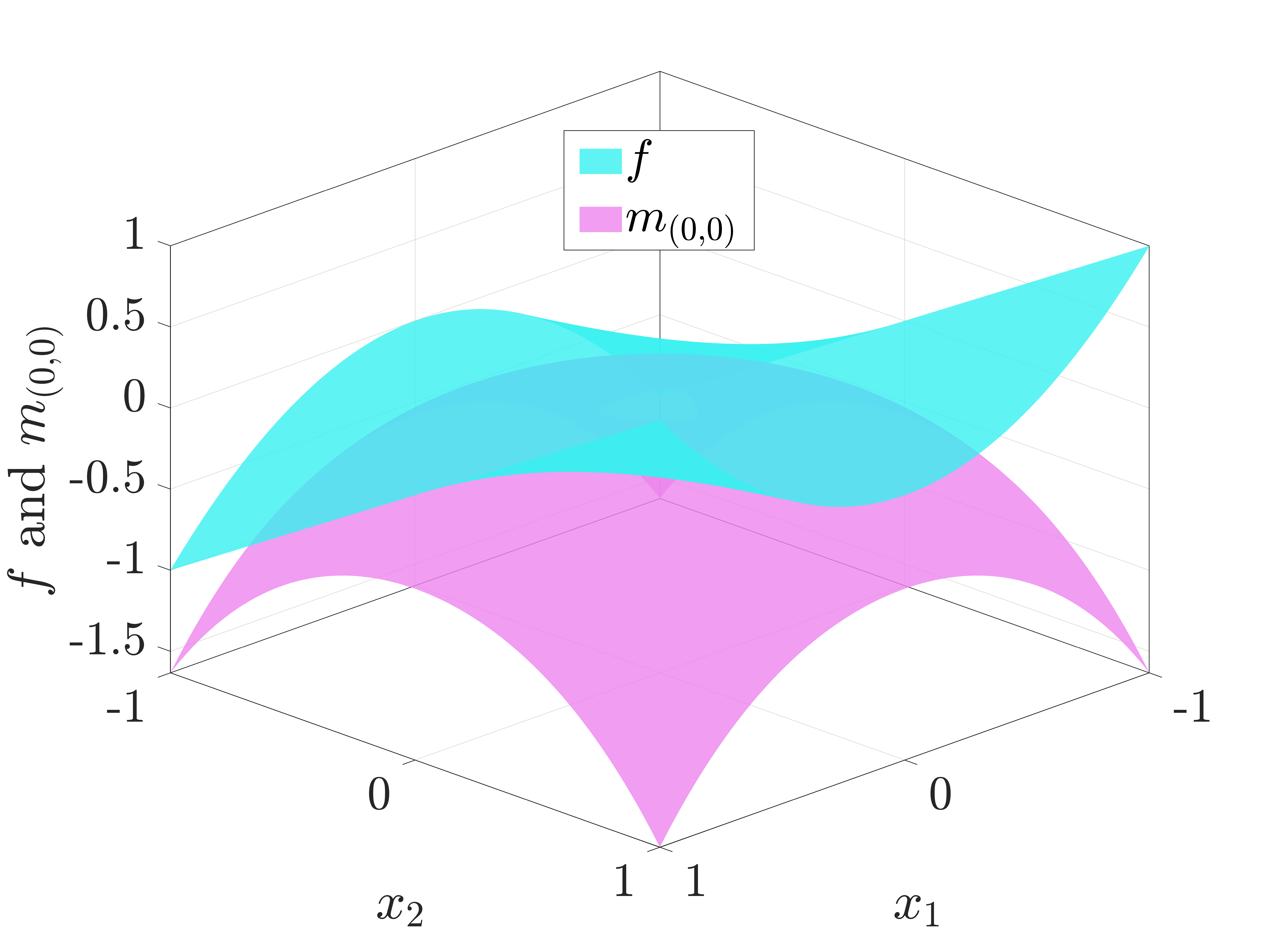}
\caption{Plot for function $f$ and its underestimator $m_{(0,0)}$ together on $P$ considered in Example \ref{ex:cubic_underestimator}}
\label{fig:objective_cecu}
\end{figure}

\section{Numerical experiments}\label{sec:numerical}
In this section, we investigate the numerical performance of the proposed CECU when used with the branch-and-bound framework described in Section \ref{sec:bnb}. The numerical experiments are performed on a collection of
130 box-constrained optimization problems from \cite{puranik2017bounds} up to 20 dimensions, and the list of problems is provided in the electronic supplementary material. The proposed method is compared with two existing relaxation approaches: the $\alpha$BB method \cite{adjiman1998global} and SOEC underestimator \cite{hasan2018edge}. For each method, we record the number of iterations, the number of explored nodes, the CPU time, and the optimality gap. A problem is considered solved when the prescribed termination criterion given in Algorithm \ref{alg:bnb} is satisfied with $\epsilon=10^{-4}$. The comparison is performed in terms of the number of branch-and-bound nodes explored and the total CPU time taken by the algorithm to satisfy the stopping condition. In addition, we examine the average CPU time required to process an explored branch-and-bound node.

\subsection{Implementation and computational environment}
\label{subsec:implementation}

All numerical experiments were conducted on a computer equipped with a 12th Gen Intel(R) Core(TM) i7-12700 2.10 GHz processor and 32 GB of RAM. The branch-and-bound framework was implemented in MATLAB R2024a.

The computational implementation consists of several components. The parameters such as $\alpha$ for $\alpha$BB \cite{adjiman1998global}, $\theta$ for SOEC \cite{hasan2018edge}, and Lipschitz constant $L$ for the proposed cubic edge-concave underestimator were computed in C++  with the help of symbolic function representation and Hessian using \texttt{GiNaC} library \cite{bauer2002introduction}. Since determining the exact Lipschitz constant of the Hessian at each branch-and-bound node can be computationally demanding, we instead employ a tractable and valid upper bound based on the Frobenius norm of the third-order derivative tensor using interval arithmetic. The resulting bound is used in the construction of the cubic underestimator. The details are provided in Section \ref{subsec:lipschitz_estimation}. Upper bounds were obtained by solving the model to local optimality using the CONOPT solver embedded within GAMS, and the corresponding incumbent solutions were passed to the branch-and-bound procedure.

The algorithm terminates when either the absolute optimality gap criterion or the relative optimality gap criterion given in Algorithm \ref{alg:bnb} is satisfied, or when the maximum number of iterations, equal to $1000$ (which corresponds to a maximum of 1999 explored nodes), is reached.  The same branch-and-bound framework is used for all three methods. In particular, the branching strategy, node selection procedure, upper-bound computation, stopping criteria, and maximum number of iterations are kept identical. The only difference between the three approaches is the construction and minimization of the underestimator used to compute the lower bound of each branch-and-bound node. Since CECU and SOEC are edge-concave, their minima are attained at the vertices of the box. Hence, the lower bound is computed by evaluating the corresponding underestimator at the vertices of the box and taking the minimum value. For $\alpha$BB method, the corresponding lower bound subproblem over a box set is solved using the MATLAB inbuilt function \texttt{fmincon}.




\subsection{Data profiles and performance measures}
We considered a common set of problems for all three methods.
To check the performance of all three methods over the complete set of considered problems, we use data profiles \cite{more2009benchmarking}.
These profiles facilitate a structured comparison of optimization algorithms by considering a set of benchmark problems $\mathcal{P}$ and a collection of algorithms $\mathcal{A}$. 
A data profile measures the absolute performance of an optimization algorithm. For each problem $p\in\mathcal{P}$ and algorithm $a\in\mathcal{A}$, let $m_{p,a}$ denote the computational effort required by algorithm $a$ to satisfy the stopping condition of branch-and-bound algorithm.
Depending on the performance measure considered, $m_{p,a}$ may represent the number of branch-and-bound nodes explored or the CPU time required to solve problem $p$. The data profile of algorithm $a$ is defined by
\begin{equation}
    d_a(\zeta) = \frac{1}{|\mathcal{P}|} \left| \left\{p\in\mathcal{P}:m_{p,a}\leq\zeta\right\} \right|,
    \label{eq:dataprofile}
\end{equation}
where $|\mathcal{P}|$ denotes the total number of benchmark problems and $\zeta$ represents a prescribed limit. Thus, $d_a(\zeta)$ gives the fraction of benchmark problems that algorithm $a$ solves within a limit of $\zeta$. If algorithm $a$ fails to satisfy the convergence criterion for problem $p$, we set
$m_{p,a}=\infty$. This means that this problem is counted as unsolved. 

We first consider a node based data profile (Figure \ref{fig_profile_nodes} (A)) by taking
$m_{p,a}=N_{p,a}$, where $N_{p,a}$ denotes the number of branch-and-bound nodes explored by algorithm $a$ for problem $p$. We also consider a CPU time based data profile (Figure \ref{fig_profile_nodes} (B)) by taking $m_{p,a}=T_{p,a}$, where $T_{p,a}$ denotes the total CPU time required by algorithm $a$ to solve problem $p$. The reported CPU time corresponds to the total wall-clock time of the branch-and-bound procedure, including the time required for the parameter computations ($\alpha,\theta$, and $\Theta$), the MATLAB implementation of the branch-and-bound algorithm, and the external GAMS calls used to obtain upper bounds. This measure reflects the overall computational effort of the branch-and-bound procedure.

In addition to the node based and CPU time data profiles, we consider the average CPU time per explored node (Figure \ref{fig_profile_nodes} (C)), defined by $T_{\mathrm{node}} =\frac{T_{\mathrm{CPU}}}{N_{\mathrm{nodes}}},$
where $N_{\mathrm{nodes}}$ denotes the number of explored branch-and-bound nodes by the algorithm. This quantity provides additional information about the computational cost of evaluating nodes, and therefore, helps distinguish between the efficiency of the underestimators and the computational cost required to construct and evaluate them.





\begin{figure}[H]
    \centering
    \includegraphics[width=\linewidth]{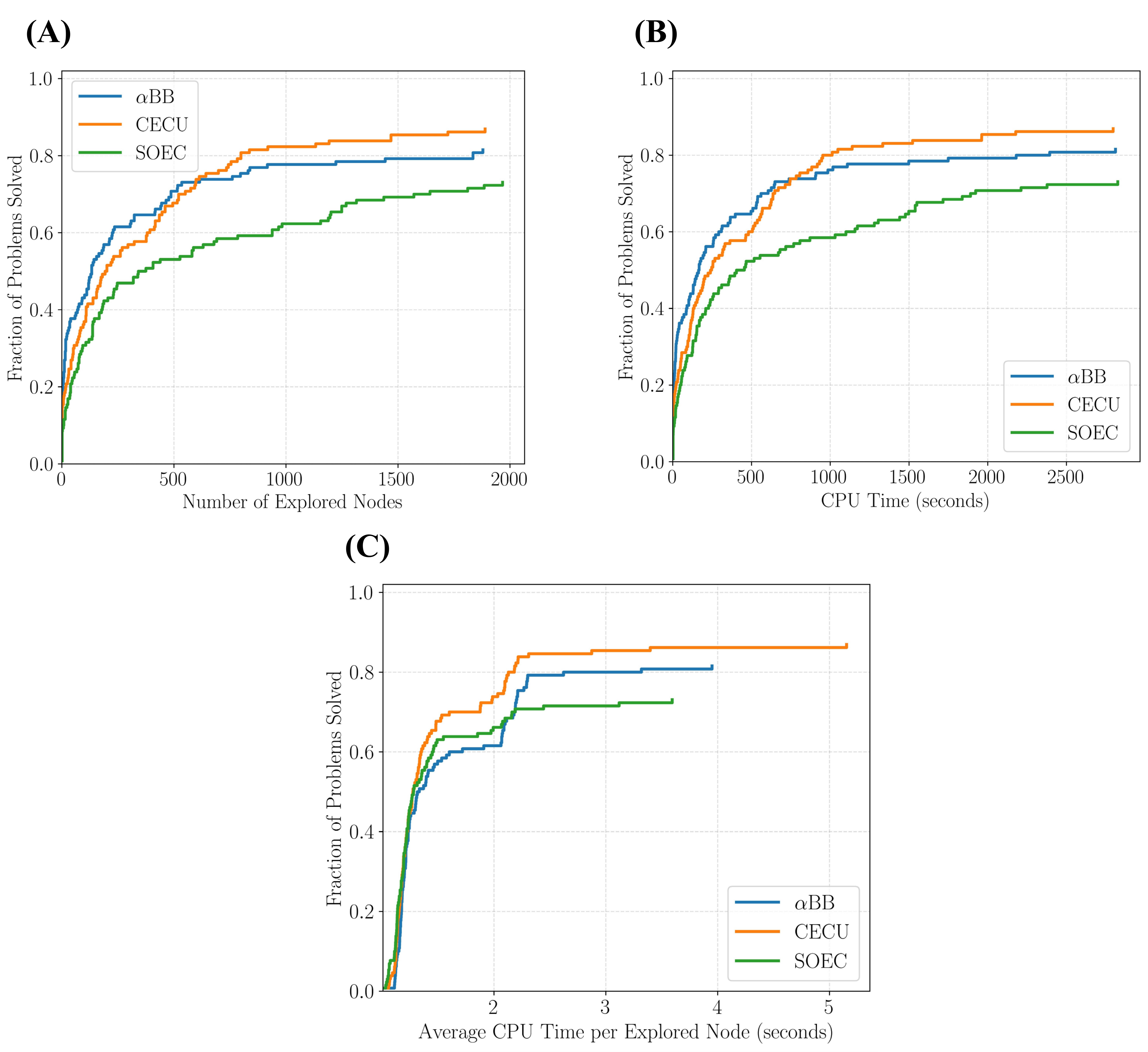}
    \caption{Data profiles for $\alpha$BB, CECU, and SOEC based on (A) explored nodes, (B)  CPU time, and (C) average CPU time per node}
    \label{fig_profile_nodes}
\end{figure}

\subsection{Discussion of numerical results}
\label{subsec:numerical_discussion}

The data profiles in Figure \ref{fig_profile_nodes} (A), (B), and (C) show the fraction of benchmark problems solved as a function of the
corresponding computational measure. In data profiles, a curve that lies higher and reaches a given fraction of solved problems at a smaller computational limit indicates better computational performance.

Figure \ref{fig_profile_nodes} (A) presents the node based data profiles. For a small node budget, $\alpha$BB initially solves a larger fraction of the benchmark problems than CECU. However, CECU outperforms $\alpha$BB as the number of nodes approaches the node limit, which indicates that CECU solves a larger number of problems within node limits than the other two methods. In particular, CECU solves
approximately $82\%$ of the considered problems within $1000$ explored nodes and approximately $85\%$ within $1500$ explored nodes. The profile reaches a final value of approximately $87\%$. In comparison, $\alpha$BB solves approximately $81\%$, while SOEC solves approximately $73\%$ of the total considered problems.

A similar behaviour is observed for the CPU time data profiles shown in Figure \ref{fig_profile_nodes} (B).  CECU solves approximately $81\%$ of the considered problems within $1000$ seconds and this percentage reaches $83\%$ within $1500$ seconds.  Only a few computationally demanding instances require substantially larger CPU time, giving rise to the tail of the CECU profile. However, $\alpha$BB
solves approximately $76\%$ and $78\%$ of the considered problems within $1000$ and $1500$ seconds, respectively, while SOEC solves approximately $59\%$ within 1000 seconds and $64\%$ within 1500 seconds.  

Figure \ref{fig_profile_nodes} (C) further examines the computational cost of all three methods in terms of the average CPU time required per explored branch-and-bound node. The CECU curve is generally shifted to the left of the $\alpha$BB and SOEC curves over the relevant range, and it ultimately attains the largest fraction of solved problems. The CECU profile reaches approximately $70\%$ of the considered problems at an average CPU time of about $1.5$ seconds per explored node and approximately $84\%$ at around $2.3$ seconds. On the other hand, $\alpha$BB reaches approximately $80\%$ at about $3.0$ seconds and approximately $81\%$ at about $3.9$ seconds, while SOEC reaches approximately $72\%$ at about $3.5$ seconds and a final value of approximately $73\%$. These results suggest that the proposed method not only benefits from a reduction in the number of explored nodes but also provides a competitive computational cost required at each node. The difference in computational time mainly comes from how the lower bounds are computed. In CECU and SOEC, the minimum over a subbox can be found by evaluating the underestimator at the vertices of the subbox, where as $\alpha$BB requires an additional optimization step to find the minimum of its convex underestimator over the subbox.

Taken together, the three data profiles demonstrate a clear computational advantage of the proposed CECU over the larger computational budgets considered. These results provide strong numerical evidence that the proposed CECU is not only effective in reducing the amount of branch-and-bound exploration but also maintains a competitive computational cost. 

\section{Conclusion}
\label{sec:conclusion}
In this paper, we proposed a new cubic edge-concave underestimator for deterministic global optimization of twice continuously differentiable functions with Lipschitz continuous Hessians. The proposed underestimator is constructed by augmenting the second-order Taylor expansion with a cubic remainder term derived from the Lipschitz bound on the Hessian. This construction provides a valid underestimator that is suitable for deterministic branch-and-bound algorithms. We derived a sufficient condition for edge-concavity on a given box, enabling lower bound computation via vertex evaluation within the branch-and-bound framework. We also analyzed the convergence properties of the proposed estimator scheme and established both pointwise and Hausdorff convergence of order three with respect to the width of the box.  

 The proposed CECU was incorporated into a deterministic branch-and-bound framework and compared numerically with the $\alpha$BB and SOEC underestimators on a collection of benchmark problems. The numerical experiments indicate that CECU can substantially reduce the number of branch-and-bound nodes required to solve the problems. CECU ultimately solves 87\% of the 130 benchmark problems within the prescribed node limit, outperforming $\alpha$BB and SOEC. Moreover, its average computational cost per explored node remains competitive with the two compared approaches.

 Future work will focus on improving the estimation of the Hessian Lipschitz constant. The proposed approach can also be extended to constrained global optimization problems. Furthermore, instead of relying exclusively on the vertex property of edge-concave functions to compute lower bounds, alternative strategies, such as constructing the polyhedral convex envelope of CECU using vertex evaluations and facet-defining hyperplanes, will be investigated  \cite{meyer2005convex}.


  \section*{Acknowledgements}
This research received support through Schmidt Sciences, LLC.
   \subsection*{Supplementary information}

Electronic Supplementary Information (ESI) is available.

\section*{Conflict of interest}

 The authors declare that they have no known competing financial interests or personal relationships
that could have appeared to influence the work reported in this paper.

\addcontentsline{toc}{section}{References}
\bibliographystyle{apalike}
\bibliography{ref}

\end{document}